\documentclass[11pt,a4paper]{article}
\usepackage[left=2.5cm, top=2.5cm,bottom=2.5cm,right=2.5cm]{geometry}
\usepackage{amsmath,amssymb,amsthm,mathrsfs,calc,graphicx}
\usepackage[british]{babel}

\newtheorem {theorem}{Theorem} 
\newtheorem {proposition}[theorem]{Proposition}
\newtheorem {lemma}[theorem]{Lemma}
\newtheorem {corollary}[theorem]{Corollary}

{\theoremstyle{definition}

}
\newtheorem {remark}[theorem]{Remark}

\newcommand{\Var}{\operatorname{Var}}

\def\ba{\begin{array}}
\def\ea{\end{array}}
\def\bea{\begin{eqnarray} \label}
\def\eea{\end{eqnarray}}
\def\be{\begin{equation} \label}
\def\ee{\end{equation}}
\def\bit{\begin{itemize}}
\def\eit{\end{itemize}}
\def\ben{\begin{enumerate}}
\def\een{\end{enumerate}}

\def\lan{\langle}
\def\ran{\rangle}

\def\CC{\mathbb{C}}

\def\EE{\mathbb{E}}

\def\PP{\mathbb{P}}

\def\RR{\mathbb{R}}
\def\SS{\mathbb{S}}

\def\g{\gamma}

\def\s{\sigma}

\def\D{\Delta}

\def\bW{\mathbf{W}}

\def\cC{\mathcal{C}}

\def\cF{\mathcal{F}}

\def\cH{\mathfrak{H}}
\def\cI{\mathcal{I}}

\def\cN{\mathcal{N}}

\def\sA{\mathscr{A}}

\def\dint{\textup{d}}

\def\s{\otimes}

\def\it{{\rm int}}
\def\cl{{\rm cl}}
\def\cum{{\rm cum}}

\begin{document}

\title{\bfseries Cumulants on Wiener chaos: moderate deviations and the fourth moment theorem}

\author{Matthias Schulte\footnotemark[1]\ \ and Christoph Th\"ale\footnotemark[2]}

\date{}
\renewcommand{\thefootnote}{\fnsymbol{footnote}}
\footnotetext[1]{Karlsruhe Institute of Technology, Department of Mathematics, D-76128 Karlsruhe, Germany. E-mail: matthias.schulte@kit.edu}

\footnotetext[2]{Ruhr University Bochum, Faculty of Mathematics, D-44780 Bochum, Germany. E-mail: christoph.thaele@rub.de}

\maketitle

\begin{abstract}
A moderate deviation principle as well as moderate and large deviation inequalities for a sequence of elements living inside a fixed Wiener chaos associated with an isonormal Gaussian process are shown. The conditions under which the results are derived coincide with those of the celebrated fourth moment theorem of Nualart and Peccati. The proofs rely on sharp estimates for cumulants. As applications, explosive integrals of a Brownian sheet, a discretized version of the quadratic variation of a fractional Brownian motion and the sample bispectrum of a spherical Gaussian random field are considered.
\end{abstract}

\section{Introduction}

In the last few years, the theory around probabilistic approximations of multiple stochastic integrals has advanced significantly. One of the cornerstones in this context is the so-called fourth moment theorem of Nualart and Peccati. To state it, let for simplicity $(A,\sA)$ be a Polish space equipped with a non-atomic $\sigma$-finite measure $\mu$ and let $(h_n:n\geq 1)$ be a sequence of symmetric, measurable and square-integrable functions on $A^q$ for a fixed integer $q\geq 2$. We assume that $h_n$ satisfies $\|h_n\|_{L^2(A^q)}=1$ for all $n\geq 1$, where $\|\,\cdot\,\|_{L^2(A^q)}$ stands for the usual norm on $L^2(A^q)$. Now, denote by $F_n=I_q(h_n)$ the multiple stochastic integral of order $q$ of $h_n$ with respect to a Gaussian random measure on $A$ with control $\mu$. The case $q=1$ is trivial, because then $F_n$ is a standard Gaussian random variable, and, thus, excluded. The fourth moment theorem (see \cite{NualartPeccati} and also \cite[Theorem 5.2.7]{NourdinPeccatiBook}) asserts that the sequence $(F_
n:n\geq 1)$ converges in distribution to a Gaussian random variable with variance $q!$ if and only if, as $n\to\infty$, the fourth cumulant of $F_n$ tends to zero, i.e.,
$$
\lim_{n\to\infty}\cum_4(F_n)=\lim_{n\to\infty}\EE[F_n^4]-3(q!)^2=0\,,
$$
or equivalently, if
\begin{equation}\label{eq:CLTCondition}
\lim_{n\to\infty}\|h_n\otimes_rh_n\|_{L^2(A^{2(q-r)})}=0\quad\text{for all }r\in\{1,\ldots,q-1\}\,.
\end{equation}
Here, $\otimes_r$ stands for the $r$th contraction operator, details and precise definitions follow below. In addition, one has the following estimate for the rate of convergence (see \cite{NourdinPeccati09} or combine Equations (5.2.6) and (5.2.13) in \cite{NourdinPeccatiBook}):
\begin{equation}\label{eq:RateAusDemBuch}
\begin{split}
d_{TV}(F_n,N) &:=\sup_{B\subset\RR\ \text{Borel set}}\big|\PP(F_n\in B)-\PP(N\in B)\big| \\ & \leq c_q\,\max_{r=1,\ldots,q-1}\|h_n\otimes_rh_n\|_{L^2(A^{2(q-r)})}\,,
\end{split}
\end{equation}
where $N$ stands for a centred Gaussian random variable with variance $q!$ and $c_q>0$ is a constant only depending on $q$. Note that the variance $q!$ comes from our normalization $\|h_n\|_{L^2(A^q)}=1$ also used below. We emphasize that in \cite{NourdinPeccatiOptimalRates} an improved and in fact optimal rate for $d_{TV}(F_n,N)$ has been derived in terms of the third and the fourth cumulant of $F_n$.

The aim of this paper is to show that under condition \eqref{eq:CLTCondition} the sequence $(F_n:n\geq 1)$ also satisfies a moderate deviation principle and fulfils moderate and large deviation inequalities. This is a direct refinement of the fourth moment theorem.  A key step in our proof is to control the growth of cumulants. Such an approach has previously been used in \cite{NourdinPeccatiCumulants} to give an alternative proof of the fourth moment theorem. To extend this to moderate deviations, we derive considerably sharper estimates for cumulants by combining classical cumulant formulas for multiple stochastic integrals with lower bounds on the matching number of regular multigraphs. We use them together with classical large deviation results of Bentkus, Rudzkis, Saulis and Statulevi\v{c}ius (see the book \cite{SaulisBuch}) and a transfer principle from the paper \cite{EichelsbacherDoering} of D\"oring and Eichelsbacher, which allows to deduce moderate deviation principles from cumulant bounds. To the best of our knowledge, the moderate deviation principle for sequences of multiple stochastic integrals is new. As applications of our general results we derive moderate deviation principles and related deviation inequalities for explosive integrals of a Brownian sheet, a discretized version of the quadratic variation of a fractional Brownian motion on the unit interval and the sample bispectrum of a Gaussian random field on the two-dimensional unit sphere.

\medskip

To motivate our results and to introduce some basic ideas, let us briefly consider the case of a sum $S_n=X_1+\ldots+X_n$ of $n$ independent and identically distributed centred random variables with variance $\sigma^2>0$. Let us assume that the random variables $(X_i: i\geq 1)$ have finite exponential moments in that $\EE[e^{\lambda X_1}]<\infty$ for all $|\lambda|\leq\Lambda$ and some $\Lambda>0$. Denoting by $\Phi_{\sigma^2}(\,\cdot\,)$ the distribution function of a centred Gaussian random variable with variance $\sigma^2$, the central limit theorem ensures that
\begin{equation}\label{CLT}
\lim_{n\to\infty}\frac{\PP({S_n/\sqrt{n}}\geq z)}{1-\Phi_{\sigma^2}(z)}=1\qquad\text{and}\qquad\lim_{n\to\infty}\frac{\PP({S_n/\sqrt{n}}\leq -z)}{\Phi_{\sigma^2}(-z)}=1
\end{equation}
for fixed $z\geq 0$. In the theory of moderate deviations one is interested in the following two questions refining the central limit theorem.
\begin{itemize}
\item[1.] How fast do the ratios in \eqref{CLT} converge to $1$?
\item[2.] Under which conditions does relation \eqref{CLT} remain valid if $z$ is growing with $n$, i.e., if $z$ is replaced by $a_nz$ with $a_n\to\infty$, as $n\to\infty$?
\end{itemize}
One way to approach moderate deviations is to provide upper bounds for the ratios in \eqref{CLT} on a logarithmic scale. Here, one can show that there are constants $n_0,z_0,c>0$ only depending on the distribution of the random variable $X_1$ such that, for all $n\geq n_0$ and $0\leq z \leq z_0 \sqrt{n}$,
$$
\Big|\log\frac{\PP(S_n/\sqrt{n}\geq z)}{1-\Phi_{\sigma^2}(z)}\Big|\leq c\,\frac{1+(z/\sigma)^3}{\sqrt{n}}
$$
(a similar result is also available for the ratio $\PP(S_n/\sqrt{n}\leq -z)/ \Phi_{\sigma^2}(-z)$), answering thereby the first question. To handle the second problem stated above, one can follow Varadhan's ideas for large deviations and study the re-scaled logarithm of the probability $\PP({S_n/\sqrt{n}}\geq a_nz)=\PP({S_n/(a_n\sqrt{n})}\geq z)$. In the present situation, it is known that
\begin{equation}\label{MPDiid}
\lim_{n\to\infty}a_n^{-2}\,\log\PP\big(S_n/(a_n\sqrt{n})\in [z,\infty)\big) = - \frac{z^2}{2\sigma^2}\,,
\end{equation}
whenever the sequence $(a_n:n\geq 1)$ satisfies $a_n\to\infty$ and $a_n/\sqrt{n}\to 0$, as $n\to\infty$. We remark that a similar relation also holds for the interval $(-\infty,-z]$. Extending \eqref{MPDiid} in a suitable way from intervals $B=[z,\infty)$ to arbitrary Borel sets $B\subset\RR$, one arrives at the usual form of a large deviation principle with speed $a_n^{-2}$ and rate function $z^2/(2\sigma^2)$ (formal definitions will be given below). Since the rate function is inherited from the central limit theorem, this large deviation principle is usually referred to as a moderate deviation principle.

\medskip



The rest of this paper is structured as follows. In Section \ref{sec:framework} we introduce the notation and our general set-up. The main results are stated in Section \ref{sec:Results}, while Sections \ref{sec:BrownianSheet}, \ref{sec:fBM} and \ref{sec:RandomFields} present the applications to Brownian sheet, fractional Brownian motion and spherical Gaussian random fields. The proofs of the main results are the content of the final Section \ref{sec:Proofs}.

\section{Preliminaries}\label{sec:framework}

Through this paper, $(\Omega,\cF,\PP)$ denotes an underlying probability space and expectation with respect to $\PP$ is indicated by $\EE$. By $L^2(\Omega,\cF,\PP)$ we denote the space of square-integrable random variables on $(\Omega,\cF,\PP)$.

\paragraph{Large and moderate deviation principles.}
Let $(X_n:n\geq 1)$ be a sequence of real-valued random variables, let $(s_n:n\geq 1)$ be a sequence of positive real numbers such that $s_n\to\infty$, as $n\to\infty$, and let $\cI:\RR\to[0,\infty]$ be a lower semi-continuous function with compact level sets such that $\cI\not\equiv 0$ and $\cI\not\equiv\infty$. We say that the sequence $(X_n:n\geq 1)$ satisfies a large deviation principle (LDP) with speed $s_n$ and rate function $\cI$ if for all Borel sets $B\subset\RR$ we have that
\begin{equation*}
\liminf_{n\to\infty} s_n^{-1}\log\PP(X_n\in B)\geq -\inf_{z\in\it(B)}\cI(z)
\end{equation*}
and
\begin{equation*}
\limsup_{n\to\infty} s_n^{-1}\log\PP(X_n\in B)\leq -\inf_{z\in\cl(B)}\cI(z)\,,
\end{equation*}
where $\it(B)$ and $\cl(B)$ stand for the interior and the closure of $B$, respectively, see Chapter 1.2 in \cite{DemboZeitouni}.

In our paper, we will deal with a special class of LDPs. To introduce them, fix a real-valued sequence $(a_n:n\geq 1)$ such that $a_n\to\infty$, as $n\to\infty$, and suppose that the random variables $(X_n:n\geq 1)$ satisfy a central limit theorem with limiting Gaussian distribution $\cN(0,\sigma^2)$ for some $\sigma^2>0$. In the following, we say that the re-scaled sequence $(a_n^{-1}X_n:n\geq 1)$ satisfies a moderate deviation principle (MDP) if it satisfies a LDP with speed $s_n=a_n^2$ and Gaussian rate function $\cI(z)=z^2/(2\sigma^2)$. Typically, a MDP is valid for a whole range of scales $(a_n:n\geq1)$.

\paragraph{Isonormal Gaussian processes and chaotic representation.}
In this paper we deal with sequences living inside a fixed Wiener chaos associated with an isonormal Gaussian process. To introduce the set-up formally, let here and through the rest of this paper $\cH$ be a real separable Hilbert space with inner product $\lan\,\cdot\,,\,\cdot\,\ran_\cH$ and norm $\|\,\cdot\,\|_\cH$. We denote for integers $q\geq 1$ by $\cH^{\otimes q}$ the $q$th tensor power and by $\cH^{\odot q}$ the $q$th symmetric tensor power of $\cH$. We supply $\cH^{\otimes q}$ with the canonical scalar product $\lan\,\cdot\,,\,\cdot\,\ran_{\cH^{\otimes q}}$ and norm $\|\,\cdot\,\|_{\cH^{\otimes q}}$, while $\cH^{\odot q}$ is equipped with the norm $\sqrt{q!}\|\,\cdot\,\|_{\cH^{\otimes q}}$. By $X=(X(h):h\in\cH)$ we indicate an isonormal Gaussian process over $\cH$ defined on our underlying probability space $(\Omega,\cF,\PP)$ and assume that $\cF=\sigma(X)$. In other words, $X$ is a family consisting of centred Gaussian random variables such that $\EE[X(h)X(h')]=\lan h,h'\ran_\cH$ for $h,h'\in\cH$. The $q$th Wiener chaos $\cC_q$ associated with $X$ is the closed linear subspace of $L^2(\Omega,\cF,\PP)$, which is generated by random variables of the form $H_q(X(h))$, where $H_q$ is the $q$th
Hermite polynomial and $h\in\cH$ satisfies $\|h\|_\cH=1$. We also put $\cC_0:=\RR$. It is well known that the mapping $h^{\otimes q}\mapsto H_q(X(h))$ can be extended to a linear isometry $I_q$ from $\cH^{\odot q}$ to $\cC_q$, cf.\ \cite[Chapter 2]{NourdinPeccatiBook}. For $q=0$ put $I_0(c)=c$ for all $c\in\RR$. In the particular case that $\cH=L^2(A)$ with a Polish space $(A,\mathcal{A})$ and a non-atomic $\sigma$-finite measure $\mu$, $I_q(\,\cdot\,)$ has an interpretation as multiple stochastic integral of order $q$ with respect to a Gaussian random measure on $A$ with control $\mu$ as discussed in the introduction, cf.\ \cite[Section 2.7.1]{NourdinPeccatiBook}.

It is a classical result in stochastic analysis (see \cite[Theorem 2.2.4]{NourdinPeccatiBook}, for example) that $L^2(\Omega,\cF,\PP)$ can be decomposed into an infinite orthogonal sum of Wiener chaoses $\cC_q$, $q\geq 0$. In particular, any $F\in L^2(\Omega,\cF,\PP)$ can be represented as
$$
F=\sum_{q=0}^\infty I_q(h^{(q)})
$$
with $h^{(0)}=\EE[F]$ and uniquely determined elements $h^{(q)}\in\cH^{\odot q}$, $q\geq 1$. We finally notice that random variables of the form $I_q(h)$ for $h\in\cH^{\odot q}$ satisfy $\EE[I_q(h)]=0$, $\EE[I_q(h)^2]=q!\|h\|_{\cH^{\otimes q}}^2$ and have finite moments of all orders due the hypercontractivity property \cite[Theorem 2.7.2]{NourdinPeccatiBook}.

\paragraph{Contractions.}
Let $(e_n: 1\leq n \leq \dim \cH)$ if $\dim \cH<\infty$ or $(e_n:n\geq 1)$ if $\dim \cH=\infty$ be a complete orthonormal system in $\cH$. For integers $p,q\geq 1$, $f\in\cH^{\odot p}$, $g\in\cH^{\odot q}$ and $r\in\{1,\ldots,\min(p,q)\}$ we denote by $f\otimes_rg$ the $r$th contraction of $f$ and $g$ defined as
$$
f\otimes_rg:=\sum_{i_1,\ldots,i_r=1}^{\dim \cH}\lan f,e_{i_1}\otimes\cdots\otimes e_{i_r}\ran_{\cH^{\otimes r}}\otimes\lan g,e_{i_1}\otimes\cdots\otimes e_{i_r}\ran_{\cH^{\otimes r}}\,,
$$
see \cite[Chapter B.4]{NourdinPeccatiBook}. We notice that $f\otimes_qg=\lan f,g\ran_{\cH^{\otimes q}}$ if $p=q$. In case that $\cH=L^2(A)$ for a Polish space $(A,\mathcal{A})$ and a non-atomic $\sigma$-finite measure $\mu$ we have that $\cH^{\odot q}=L^2_s(A^q)$ (that is the subspace of $\mu^q$-a.e.\ symmetric functions in $L^2(A^q))$ and that
\begin{align*}
f\s_rg(a_1,\ldots,a_{p+q-2r})=\int_{A^r} &f(x_1,\ldots,x_r,a_1,\ldots,a_{p-r})\\
&\qquad \times g(x_1,\ldots,x_r,a_{p-r+1},\ldots,a_{p+q-2r})\,\mu^r(\dint(x_1,\ldots,x_r))
\end{align*}
for $f\in L_s^2(A^p)$, $g\in L_s^2(A^q)$ and $r\in\{1,\ldots,\min(p,q)\}$. In other words $f\s_rg$ is a function of $p+q-2r$ arguments, which arises from the tensor product of $f$ and $g$ by identifying $r$ variables which are then integrated out.

\paragraph{Cumulants.}
Let $F$ be a real-valued random variable such that $\EE[|F|^m]<\infty$ for some integer $m\geq 1$. By $\phi_F(t):=\EE[e^{\mathfrak{i}tF}]$, $t\in\RR$, we denote the characteristic function of $F$, where $\mathfrak{i}$ is the imaginary unit. Then the $m$th cumulant of $F$ (sometimes also called semi-invariant) is defined as
\begin{equation}\label{eq:DefCumulant}
\cum_m(F):=(-\mathfrak{i})^m\frac{\dint^m}{\dint t^m}\log\phi_F(t)\Big|_{t=0}\,.
\end{equation}
For example, $\cum_1(F)$ coincides with the mean of $F$, while $\cum_2(F)$ is its variance. Moreover, for centred random variables $F$ we have the relations $\cum_3(F)=\EE[F^3]$ and $\cum_4(F)=\EE[F^4]-3(\EE[F^2])^2$.

\section{Main results}\label{sec:Results}

Let $\cH$ be a real separable Hilbert space underlying our isonormal Gaussian process and $(h_n:n\geq 1)$ be a sequence of elements of $\cH^{\odot q}$ for some fixed integer $q\geq 2$. To simplify some of our arguments below, we assume without loss of generality that $\|h_n\|_{\cH^{\otimes q}}=1$ for each $n\geq 1$. This implies that $\Var I_q(h_n)=\EE[I_q(h_n)^2]=q!\|h_n\|_{\cH^{\otimes q}}^2=q!$ for all $n\geq 1$. We further define, for $n\geq 1$,
\begin{equation}\label{eq:defKn}
K_n:=\max\limits_{r=1,\ldots,q-1}\|h_n\otimes_rh_n\|_{\cH^{\otimes 2(q-r)}}
\end{equation}
and put
\begin{equation}\label{eq:defAlphaq}
\alpha(q) := \frac{q+2}{3q+2}\quad \text{($q$ even)}\,,\qquad \alpha(q) := \frac{q^2-q-1}{q(3q-5)}\quad \text{($q$ odd)}\,.
\end{equation}
Our first result delivers a moderate deviation principle (MDP) for the sequence $(I_q(h_n):n\geq 1)$ in the regime in which $K_n\to 0$, as $n\to\infty$. As discussed in the introduction, this is precisely the situation under which the fourth moment theorem ensures that the sequence $(I_q(h_n):n\geq 1)$ satisfies a central limit theorem, recall \eqref{eq:RateAusDemBuch}.

\begin{theorem}\label{thm:MDP}
Define $F_n:=I_q(h_n)$ and suppose that $\lim\limits_{n\to\infty}K_n=0$. Further, let $(a_n:n\geq 1)$ be a real sequence such that
$$\lim\limits_{n\to\infty}a_n=\infty\qquad\text{and}\qquad\lim\limits_{n\to\infty}a_n/\Delta_n^{1/(q-1)}=0$$ with $\Delta_n= (q^{3q/2}K_n^{\alpha(q)})^{-1}$. Then the sequence $(a_n^{-1}F_n:n\geq 1)$ satisfies a MDP with speed $a_n^2$ and Gaussian rate function $\cI(z)=z^2/(2q!)$.
\end{theorem}

Theorem \ref{thm:MDP} can be seen as a direct refinement of the fourth moment theorem of Nualart and Peccati \cite{NualartPeccati}. It provides under exactly the same conditions further information about the distributional behaviour of the involved random variables in form of a MDP.

Our next result provides a version of Theorem \ref{thm:MDP}, which is more amenable to some concrete applications, see Section \ref{sec:RandomFields} for an example. It shows that the sequence $(F_n:n\geq 1)$ satisfies a MDP if $\cum_4(F_n)\to 0$, as $n\to\infty$. Defining
\begin{equation}\label{eq:DefinitionLn}
L_n:=(q\,q!)^{-1}\,\sqrt{\cum_4(F_n)}\,,\qquad n\geq 1\,,
\end{equation}
we see that the following result is a direct consequence of Theorem \ref{thm:MDP} and the estimate $K_n\leq L_n$ from \cite[Equation (4.5)]{BiermeEtAl}.

\begin{corollary}\label{cor:MDP}
Let $F_n:=I_q(h_n)$ and suppose that $\lim\limits_{n\to\infty}\cum_4(F_n)=0$. Then the conclusion of Theorem \ref{thm:MDP} remains valid with $K_n$ replaced by $L_n$ in the definition of $\Delta_n$.
\end{corollary}

Note that the moderate deviation principle in Theorem \ref{thm:MDP} or Corollary \ref{cor:MDP} for the sequence of random variables $F_n=I_q(h_n)$, $n\geq 1$, holds in a range of scales $(a_n:n\geq 1)$, which shrinks with growing $q$. The following example shows that this phenomenon is unavoidable. For this, recall that $H_q$ stands for the $q$th Hermite polynomial and that the random variables of the form $H_q(X(h))=I_q(h^{\otimes q})$ with $h\in\cH$ and $\|h\|_{\cH}=1$ are the basic building blocks of the $q$th Wiener chaos $\cC_q$.

\begin{proposition}\label{prop:SumHermite}
Suppose that $\dim \cH=\infty$, let $(e_k:k\geq 1)$ be a complete orthonormal system in $\cH$, fix $q\geq 2$ and for each $n\geq 1$ define $F_n:=\frac{1}{\sqrt{n}}\sum_{k=1}^n H_q(X(e_k))$. Then $F_n\in\cC_q$, $\Var F_n =q!$, $K_n=1/\sqrt{n}$ and there are constants $C,c,z_0>0$ such that, for any $n\geq 1$,
\begin{equation}\label{eq:LowerBoundHermite}
\PP(|F_n| \geq z) \geq C \exp(-c\, n^{1/q} z^{2/q}) \,, \quad z\geq z_0\,.
\end{equation}
Moreover, if $(a_n: n\geq 1)$ is a real sequence such that
\begin{equation}\label{eq:RangeNoMDP}
\lim_{n\to\infty} a_n = \infty \quad \text{ and } \quad \lim_{n\to\infty} {a_n}/{n^{1/(2q-2)}}=\infty \,,
\end{equation}
then the sequence $(a_n^{-1} F_n: n\geq 1)$ does not satisfy a MDP with speed $a_n^2$ and Gaussian rate function $\cI(z)=z^2/(2q!)$.
\end{proposition}

The random variables $(F_n: n\geq 1)$ considered in Proposition \ref{prop:SumHermite} satisfy the assumptions of Theorem \ref{thm:MDP}. It asserts that for a real sequence $(a_n:n\geq 1)$ the random variables $(a_n^{-1}F_n: n\geq 1)$ satisfy a MDP with speed $a_n^2$ and Gaussian rate function $\cI(z)=z^2/(2q!)$ if
\begin{equation}\label{eq:RangeMDP}
\lim_{n\to\infty} a_n = \infty \quad \text{ and } \quad \lim_{n\to\infty} {a_n}/{n^{\alpha(q)/(2q-2)}}=0\,.
\end{equation}
Comparing conditions \eqref{eq:RangeNoMDP} and \eqref{eq:RangeMDP}, we see that there are re-scalings $(a_n:n\geq 1)$ such that the sequence $(a_n^{-1}F_n:n\geq 1)$ of random variables from Proposition \ref{prop:SumHermite} satisfies a MDP and re-scalings $(a_n:n\geq 1)$ for which no such MDP can hold. If the growth of $(a_n:n\geq 1)$ is between that of $n^{\alpha(q)/(2q-2)}$ and $n^{1/(2q-2)}$, it remains an open problem whether $(a_n^{-1} F_n: n\geq 1)$ satisfies a MDP or not. We emphasize that via the method of cumulants -- which is the key tool to derive Theorem \ref{thm:MDP} -- it is not possible to decide this question. More precisely, our proof of Theorem \ref{thm:MDP} relies on a combination of cumulant estimates, lower bounds for the matching number of regular multigraphs and the general theory of large deviations from \cite{SaulisBuch} and its recent extension \cite{EichelsbacherDoering}. We only need to bound the cumulants of $F_n$ from above. However, since there are situations in which all of our estimates in the proof of Theorem \ref{thm:MDP} are sharp, we see that our cumulant bounds can in general not be improved as further discussed in Remark \ref{rem:lowerBound}. This in turn implies that the range of re-scalings $(a_n:n\geq 1)$ of the MDP in Theorem \ref{thm:MDP} or Corollary \ref{cor:MDP} is the best one can achieve by the method of cumulants. We remark that cumulant estimates for multiple stochastic integrals have also been performed in the proof of Theorem 5.8 in \cite{NourdinPeccatiCumulants}. While these bounds are sufficient to yield the fourth moment theorem, they are far too crude in order to show our results.

\begin{remark}\rm
Usually a large deviation principle is called a moderate deviation principle if the magnitude of re-scaling is between that of the central limit theorem and that of a law of large numbers. For the random variables in Theorem \ref{thm:MDP} the latter one does not apply since in general there is no law of large numbers. However, for $a_n=\sqrt{n}$ the random variables $(a_n^{-1}F_n: n\geq 1)$ in Proposition \ref{prop:SumHermite} satisfy the classical strong law of large numbers, which justifies the denotation as MDP.
\end{remark}

Our next theorem deals with moderate and large deviation probabilities. To state it, recall that $\Phi_{\sigma^2}$ denotes the distribution function of a centred Gaussian random variable with variance $\sigma^2$.

\begin{theorem}\label{thm:LDIs}
Let $F_n:=I_q(h_n)$ and put $\Delta_n:=(q^{3q/2}\,K_n^{\alpha(q)})^{-1}$.
\begin{itemize}
\item[(i)] There are constants $c_0,c_1,c_2>0$ depending only on $q$ such that, for $\D_n\geq c_0$ and $0\leq z\leq c_1\D_n^{1/(q-1)}$,
$$
\bigg|\log\frac{\PP(F_n\geq z)}{1-\Phi_{q!}(z)}\bigg|\leq c_2 \frac{1+(z/\sqrt{q!})^3}{\D_n^{1/(q-1)}}\qquad{\rm and}\qquad\bigg|\log\frac{\PP(F_n\leq -z)}{\Phi_{q!}(-z)}\bigg|\leq c_2\frac{1+(z/\sqrt{q!})^3}{\D_n^{1/(q-1)}}\,.
$$
\item[(ii)] For all $z\geq 0$ one has
\begin{equation}\label{eq:LDI}
\PP(|F_n|\geq z)\leq 2\exp\Big(-\frac{1}{4}\min\Big\{\frac{z^2}{2^{q/2}},(z\Delta_n)^{2/q}\Big\}\Big)\,.
\end{equation}
\item[(iii)] The statements (i) and (ii) remain valid if in the definition of $\Delta_n$, $K_n$ is replaced by $L_n$ given by \eqref{eq:DefinitionLn}.
\end{itemize}
\end{theorem}

\begin{remark}\label{rem:DiscussionWoher}\rm
\begin{itemize}
\item[(i)] The first part of Theorem \ref{thm:LDIs} is a consequence of a version of the celebrated `lemma' of Rudzkis, Saulis and Statulevi\v{c}ius \cite{RudzkisSaulisStatu} applied to $F_n$. Its original formulation involves the so-called Cram\'er-Petrov series. For clarity and to avoid heavy notation, we have decided to state here the result in a simplified form taken from \cite[Corollary 3.2]{EichelsbacherSchreiberRaic}, suppressing thereby higher-order terms of the expansion.
\item[(ii)] The second part of Theorem \ref{thm:LDIs} follows from a result of Bentkus and Rudzkis \cite{BentkusRudzkis}. Our statement is a simplified version taken from the Corollary after Lemma 2.4 in \cite{SaulisBuch}.
\item[(iii)] As Corollary \ref{cor:MDP}, Theorem \ref{thm:LDIs} (iii) is again a consequence of the estimate $K_n\leq L_n$.
\end{itemize}
\end{remark}
Let us relate our result obtained in Theorem \ref{thm:LDIs} (ii) to the existing literature, especially to the tail estimates for multiple stochastic integrals established by Major \cite{Major}. Theorem 8.5 there says that there is a constant $c>0$ only depending on $q$ such that $F_n=I_q(h_n)$ with $\|h_n\|_{\cH^{\otimes q}}=1$ satisfies
\begin{equation}\label{eq:Major}
\PP(|F_n|>z)\leq c\,\exp\Big(-\frac{1}{2}\Big(\frac{z}{\sqrt{q!}}\Big)^{2/q}\,\Big)\qquad\text{for all $z\geq 0$}\,.
\end{equation}
We also refer to \cite[Theorem 6.12]{Janson} for a similar result. If $z$ and $\Delta_n$ are sufficiently large, the minimum in \eqref{eq:LDI} is larger than $(z/\sqrt{q!})^{2/q}/2$, whence \eqref{eq:LDI} yields better estimates than \eqref{eq:Major}. This is the case if $K_n$ is small, meaning in view of \eqref{eq:RateAusDemBuch} that the distribution of $F_n$ is close to a Gaussian distribution.

Moreover, we would like to mention that Chapter 5.2 of \cite{SaulisBuch} also contains a set of large deviation inequalities for multiple stochastic integrals. These results only involve the $L^2$-norms of the integrands and not the contractions as encoded in the sequence $K_n$ defined at \eqref{eq:defKn} so that a connection to the central limit theorem remained hidden. Applying the Cauchy-Schwarz inequality to $K_n$, we can recover the results in \cite{SaulisBuch} from our Theorem \ref{thm:LDIs}. We also point out that in \cite{SaulisBuch} no connection has been made to a MDP as stated in Theorem \ref{thm:MDP}.

\section{Application to Brownian sheet}\label{sec:BrownianSheet}

Let $W=(W_t:t\in[0,1])$ be a standard Brownian motion on the unit interval. Then
$$
\EE\int_0^1 \frac{W_t^2}{t^2}\,\dint t=\int_0^1 \frac{\EE W_t^2}{t^2}\,\dint t=\int_0^1\frac{1}{t}\,\dint t=\infty \,,
$$
and Jeulin's Lemma (see \cite{PeccatiYor}) implies that $\int_0^1 \frac{W_t^2}{t^2}\,\dint t=\infty$ with probability one. However, for any $n\geq 2$ the functional
\begin{equation}\label{eq:DefG_n}
G_n:=\int_{1/n}^1 \frac{W_t^2}{t^2}\,\dint t
\end{equation}
has mean $\EE[G_n]=\log n$ and is finite with probability one. It is thus natural to describe the rate and type of `explosion' of the random integral, as $n\to\infty$. We emphasize that such explosive random integrals have some fundamental connections with the theory of enlargement of filtrations and to Brownian local times as further discussed in \cite{PeccatiYor}, see also \cite{Jeulin}.


More generally, we consider a similar family of random explosive integrals with respect to a standard Brownian sheet $\bW=(\bW(t_1,\ldots,t_d):(t_1,\ldots,t_d)\in[0,1]^d)$ on the $d$-dimensional unit cube $[0,1]^d$ for some fixed space dimension $d\geq 1$. Recall that $\bW$ is a centred Gaussian random field on $[0,1]^d$ with covariance
$$\EE[\bW(t_1,\ldots,t_d)\bW(s_1,\ldots,s_d)]=\prod_{i=1}^d\min(t_i,s_i)\,,\qquad t_1,\ldots,t_d,s_1,\ldots,s_d\in[0,1]\,.$$
In analogy to \eqref{eq:DefG_n} put, for $n\geq 2$,
$$
G_n^{(d)}:=\int_{[1/n,1]^d}\frac{\bW(t_1,\ldots,t_d)^2}{ t_1^2\cdots t_d^2}\,\dint(t_1,\ldots,t_d)
$$
so that $G_n^{(1)}$ reduces to $G_n$. Taking expectation and using Fubini's theorem yields that
\begin{align*}
\EE[G_n^{(d)}] &= \int_{[1/n,1]^d}\frac{\EE[\bW(t_1,\ldots,t_d)^2]}{ t_1^2\cdots t_d^2}\,\dint(t_1,\ldots,t_d) \\
&= \int_{[1/n,1]^d}\frac{1}{t_1\cdots t_d}\,\dint(t_1,\ldots,t_d)= \Big(\int_{1/n}^1 \frac{1}{t}\,\dint t\Big)^d = (\log n)^d\,.
\end{align*}
As shown in Section 3.2 in \cite{NualartPeccati}, $G_n^{(d)}-\EE[G_n^{(d)}]$ can be represented as $I_2^\bW(h_n^{(d)})$ with
$$
h_n^{(d)}(t_1,\ldots,t_d,s_1,\ldots,s_d)=\prod_{i=1}^d\big(\max(t_i,s_i,1/n)^{-1}-1)\big)\,,
$$
where $I_2^{\bW}$ stands for the double stochastic integral with respect to the Brownian sheet $\bW$. Hence, we have that
\begin{align*}
\Var G_n^{(d)} & = 2 \int_{([0,1]^{d})^2} h_n^{(d)}(t_1,\hdots,t_d,s_1,\hdots,s_d)^2 \, \dint(t_1,\hdots,t_d,s_1,\hdots,s_d)\\
& = 2\bigg(\int_0^1 \int_0^1 \big(\max\{t,s,1/n\}^{-1}-1\big)^2 \, \dint t \, \dint s \bigg)^d\\
& = 2\bigg( 2 \int_0^1 \int_0^s \big(\min\{1/t,1/s,n\}-1\big)^2 \, \dint t \, \dint s \bigg)^d\\
& = 2\bigg( 2 \int_{1/n}^1 s \Big(\frac{1}{s}-1\Big)^2\, \dint s + 2 \int_0^{1/n} s (n-1)^2 \, \dint s \bigg)^d\\
&= 2\bigg( 2 \log n -2\Big(1-\frac{1}{n}\Big)  \bigg)^d\,.
\end{align*}
We now define a normalized version of $G_n^{(d)}$ as
\begin{align}\label{eq:DefFnBrownianSheet}
F_n^{(d)}:=\frac{G_n^{(d)}-(\log n)^d}{(2\log n-2(1-1/n))^{d/2}}\,,\qquad n\geq 2\,.
\end{align}
This normalization ensures that $\EE[F_n^{(d)}]=0$ and $\Var(F_n^{(d)})= 2$.

By Proposition 8 in \cite{NualartPeccati} the random variables $(F_n^{(d)}:n\geq 2)$ satisfy a central limit theorem, as $n\to\infty$. Our theory developed in Section \ref{sec:Results} allows to add a moderate deviation principle as well as moderate and large deviation inequalities.

\begin{theorem}
Let $F_n^{(d)}$ be defined as at \eqref{eq:DefFnBrownianSheet}. Then the statements of Theorem \ref{thm:MDP} and Theorem \ref{thm:LDIs} are valid with $q=2$ and $\Delta_n=\frac{1}{8\sqrt{2}}\big(\frac{\log n}{120}\big)^{d/4}$, $n\geq 2$.
\end{theorem}
\begin{proof}
In what follows, we compute an upper bound for
\begin{align*}
K_n & =\frac{2}{\Var G_n^{(d)}} \|h_n^{(d)}\otimes_1 h_n^{(d)}\|_{L^2(([0,1]^d)^2)}\\
 & =\frac{2}{\Var G_n^{(d)}} \bigg(\int_0^1 \int_0^1 \bigg(\int_0^1 h_n(x,z) h_n(y,z) \, \dint z \bigg)^2 \, \dint x \, \dint y \bigg)^{d/2}\,.
\end{align*}
For this, we make use of the estimate
$$
\int_0^1 h_n(x,z) h_n(y,z) \, \dint z \leq \int_0^1 \min(1/z,1/x,n)\, \min(1/z,1/y,n) \, \dint z=: g(x,y)\,,
$$
valid for all $x,y\in[0,1]$. We now consider the cases $1/n\leq x\leq y$, $x\leq 1/n\leq y$ and $x\leq y \leq 1/n$ separately and obtain that
\begin{align*}
g(x,y) & = \int_0^x \frac{1}{x y} \, \dint z + \int_x^y \frac{1}{zy} \, \dint z +  \int_y^1 \frac{1}{z^2} \, \dint z = \frac{1}{y}+\frac{\log y -\log x}{y}+\frac{1}{y}-1 \leq \frac{2}{y}+\frac{\log y -\log x}{y}\,,\\
g(x,y) & =\int_0^{1/n} \frac{n}{y} \, \dint z + \int_{1/n}^y \frac{1}{z y}\, \dint z + \int_y^1 \frac{1}{z^2} \, \dint z = \frac{1}{y}+\frac{\log y +\log n}{y}+\frac{1}{y}-1\leq \frac{2}{y}+\frac{\log y +\log n}{y}\,,\\
g(x,y) & = \int_0^{1/n} n^2 \, \dint z + \int_{1/n}^1 \frac{1}{z^2} \, \dint z = n+n-1\leq 2n\,,
\end{align*}
respectively. Together with the Cauchy-Schwarz inequality this implies that
\begin{align*}
\int_0^1 \int_0^1 g(x,y)^2 \, \dint x \, \dint y & \leq 2 \int_{1/n}^1 \int_{1/n}^y  \frac{8}{y^2} + \frac{2(\log y -\log x)^2}{y^2} \, \dint x\, \dint y\\
 & \quad + 2 \int_{1/n}^1 \int_0^{1/n}  \frac{8}{y^2} + \frac{2(\log y +\log n)^2}{y^2} \, \dint x\, \dint y+2 \int_0^{1/n} \int_0^y 4n^2 \, \dint x \, \dint y\\
 & \leq  2 \int_{1/n}^1 \int_0^{y}  \frac{8}{y^2} + \frac{2(\log y -\log x)^2}{y^2} \, \dint x\, \dint y+2 \int_0^{1/n} \int_0^y 4n^2 \, \dint x \, \dint y \,.
\end{align*}
Moreover, we have that
$$
\int_{1/n}^1 \int_0^y  \frac{8}{y^2} \, \dint x\, \dint y= \int_{1/n}^1  \frac{8}{y} \, \dint y = 8 \log n \quad \text{ and }\quad \int_0^{1/n} \int_0^y 4n^2 \, \dint x \, \dint y = 2
$$
as well as
\begin{align*}
\int_{1/n}^1 \int_0^y  \frac{2(\log y -\log x)^2}{y^2} \, \dint x\, \dint y 
  & = \int_{1/n}^1   \frac{4}{y} \, \dint y = 4\log n\,.
\end{align*}
Combining these estimates and using that $\frac{1}{2}\log n\leq 2\log n-2(1-1/n)$, we obtain that, for $n\geq 2$,
$$
K_n \leq 2\,\frac{(4 +24 \log n)^{d/2}}{(2\log n-2(1-1/n))^{d}} \leq 2\,\frac{(30 \log n)^{d/2}}{((\log n) /2)^d}\leq 2\,\bigg(\frac{120}{\log n}\bigg)^{d/2}
$$
and thus we can choose
$$
\Delta_n =\frac{1}{8\sqrt{2}}\Big(\frac{\log n}{120}\Big)^{d/4}\,,
$$
since $\alpha(2)=1/2$. The result then follows from Theorem \ref{thm:MDP} and Theorem \ref{thm:LDIs}.
\end{proof}

\section{Application to fractional Brownian motion}\label{sec:fBM}

We now present our second application of the results obtained in Section \ref{sec:Results} by considering a discretized version of the quadratic variation of a fractional Brownian motion. Recall that a fractional Brownian motion $B^H=(B_t^H:t\geq 0)$ with Hurst index $0<H<1$ is a continuous-time centred Gaussian process with covariance $$\EE[B_t^HB_s^H]=\frac{1}{2}\big(t^{2H}+s^{2H}-|t-s|^{2H}\big)\,,\qquad s,t\geq 0\,.$$ If $H=1/2$, then $B^H$ is the ordinary Brownian motion, while for $H>1/2$ the fractional Brownian motion is a commonly used model for long-range dependencies, see \cite{NourdinfBM} for details and background material. In practice, it is crucial to estimate $H$ from given data.  A well known estimator is based on the discretized quadratic variation of $B^H$ at scale $1/n$ on the interval $[0,1]$ and is defined as
$$S_n:=\sum_{k=0}^{n-1}(B_{\frac{k+1}{n}}^H-B_{\frac{k}{n}}^H)^2\,,\qquad n\geq 1\,.$$
From \cite[Equation (2.12)]{NourdinfBM} it is known that, as $n\to\infty$, the random variables $n^{2H-1}S_n$, $n\geq1$, converge in probability to $1$ so that a reasonable estimator $\widehat{H}_n$ for the Hurst index $H$ is given by $$\widehat{H}_n=\frac{1}{2}-\frac{\log S_n}{2\log n}\,, \qquad n\geq 2\,.$$
To investigate the asymptotic distributional behaviour of $\widehat{H}_n$, define a sequence $(F_n:n\geq 1)$ of centred and normalized versions of the discretized quadratic variation of $B^H$ by
\begin{equation}\label{eq:defFnFBM}
F_n:=\frac{n^{2H}}{\sigma_n}\sum_{k=0}^{n-1}\big[(B_{\frac{k+1}{n}}^H-B_{\frac{k}{n}}^H)^2-n^{-2H}\big]\,,\qquad n\geq 1\,,
\end{equation}
where we choose $\sigma_n$ in such a way that $\EE[F_n^2] = 2$ (this normalization is adapted to the set-up of Section \ref{sec:Results}). Now, a short computation reveals that
$$
\widehat{H}_n-H = -\frac{\log(\sigma_n F_n/n+1)}{2\log n}\,,
$$
which means that the behaviour of $F_n$ controls the error of the estimator $\widehat{H}_n$. For this reason $F_n$ is studied in the sequel.

We notice that $F_n$ has the same law as
\begin{equation}\label{eq:defFnFBMAlternative}
\frac{1}{\sigma_n}\sum_{k=0}^{n-1}\big[(B_{k+1}^H-B_k^H)^2-1\big]=\frac{1}{\sigma_n}\sum_{k=0}^{n-1}H_2(B_{k+1}^H-B_k^H)\,,
\end{equation}
where $H_2(x)=x^2-1$ is the second Hermite polynomial, explaining the alternative name second Hermite power variation for $F_n$. Asymptotic normality for $F_n$ together with rates of convergence for the total variation distance has been investigated in literature, see Theorem 6.3 in \cite{NourdinfBM} and Section 7.4 in \cite{NourdinPeccatiBook}. In fact, if $H>3/4$, then the sequence $(F_n:n\geq 1)$ does not satisfy a central limit theorem, while for $0<H\leq 3/4$ it holds that
$$
d_{TV}(F_n,N)=\sup_{B\subset\RR\ \text{Borel set}}\big|\PP(F_n\in B)-\PP(N\in B)\big|\leq A_n\,, \qquad n\geq 2\,,
$$
where $N$ is a centred Gaussian random variable with variance $2$ and $A_n$ is given by
\begin{align}\label{eq:fBMRate}
 A_n:=c_H\times\begin{cases}\frac{1}{\sqrt{n}} &: 0<H<5/8\\ \frac{(\log n)^{3/2}}{\sqrt{n}} &: H=5/8\\ \frac{1}{n^{3-4H}} &: 5/8<H<3/4\\ \frac{1}{\log n} &: H=3/4\end{cases}
\end{align}
with a constant $c_H>0$ only depending on $H$. As a consequence, one can show that for $0<H\leq 3/4$ both of the re-scaled random variables $\sqrt{n}(n^{2H-1}S_n-1)$ and $\sqrt{n}\log n\;(\widehat{H}_n-H)$ are, as $n\to\infty$, normally distributed with explicitly known limiting variances, cf.\ \cite[Chapter 6.4]{NourdinfBM}.

We are now going to study the normalized versions $F_n$ of the discretized quadratic variation functionals in more detail. For this, we will use the representation \eqref{eq:defFnFBMAlternative} for $F_n$. The connection between the random variables $F_n$ and the random elements living inside a Wiener chaos of fixed order is that, for each $n\geq 1$, $F_n$ can be represented as
$$
F_n=I_2^W(h_n)\qquad\text{with}\qquad h_n=\frac{1}{\sigma_n}\sum_{k=0}^{n-1}f_k\otimes f_k\,.
$$
Here, $I_2^W$ indicates the double stochastic integral with respect to a two-sided standard Brownian motion $W$ on $\RR$ and $(f_k:k\geq 1)$ is a certain sequence of square-integrable functions on $\RR$ such that $I_1^W(f_k)$ has the same distribution as $B_{k+1}^H-B_k^H$ for all $k$ (the precise form of the $f_k$'s is irrelevant for our purposes). This follows from the Mandelbrot-Van Ness representation of the fractional Brownian motion as stochastic integral with respect to the ordinary two-sided Brownian motion $W$, see
Proposition 2.3 in \cite{NourdinfBM}. Our next result shows that the sequence $(F_n:n\geq 1)$ of discretized and normalized quadratic variations of $B^H$ satisfies a MDP as well as certain moderate and large deviation inequalities.

\begin{theorem}\label{thm:fBM}
Suppose that $0<H\leq 3/4$ and let $F_n$ be as in \eqref{eq:defFnFBM}. Then the statements in Theorem \ref{thm:MDP} and Theorem \ref{thm:LDIs} hold with $q=2$ and $\Delta_n=2^{-9/4}A_n^{-1/2}$, where $A_n$ is given by \eqref{eq:fBMRate}.
\end{theorem}
\begin{proof}
Since $q=2$ in the language of Section \ref{sec:Results}, we have that $K_n$ defined at \eqref{eq:defKn} is just the norm of a single contraction, namely $$
K_n=\|h_n\otimes_1h_n\|_{L^2(\RR^2)}:=\sqrt{\int_{\RR^2}(h_n\otimes_1 h_n)(x,y)^2\,\dint(x,y)}\,,
$$
where the integration is with repspect to the Lebesgue measure on $\RR^2$.
In the proof of Theorem 6.3 in \cite{NourdinfBM} it has been shown that $\|h_n\otimes_1h_n\|_{L^2(\RR^2)}$ and hence $K_n$ can be estimated from above by
$$
K_n\leq \frac{1}{2\sqrt{2}}\,A_n
$$
with $A_n$ defined at \eqref{eq:fBMRate}. Since $\alpha(2)=1/2$, the choice
$\Delta_n=2^{-9/4}A_n^{-1/2}$ as well as Theorem \ref{thm:MDP} and Theorem \ref{thm:LDIs} yield the result.
\end{proof}

\begin{remark}\rm
One can more generally consider the higher-order Hermite power variations defined for $q\geq 3$ as
$$
F_n^{(q)}:=\frac{1}{\sigma_n^{(q)}}\sum_{k=0}^{n-1}H_q(B_{k+1}^H-B_k^H)\,,\qquad n\geq 1\,,
$$
where $H_q$ is the $q$th Hermite polynomial and where $\sigma_n^{(q)}$ is such that $\EE[F_n^{(q)}]=q!$ for all $n\geq 1$. These functionals can be represented as elements of the $q$th Wiener chaos and an estimate for $K_n$ can in this case be deduced from Theorem 1.2 in \cite{BretonNourdin} and Theorem 4.1 in \cite{NourdinPeccati09}, see also Exercise 7.5.1 in \cite{NourdinPeccatiBook}. Since the results for $q=2$ and $q\geq 3$ have different structures, we decided to restrict to the first case.
\end{remark}

\section{Application to spherical Gaussian random fields}\label{sec:RandomFields}

In this section we present another application of Theorem \ref{thm:MDP} and Theorem \ref{thm:LDIs} by considering the sample bispectrum of spherical random fields. These objects have recently found considerable attention especially in astrophysics, cosmology, medical imaging and geophysics, and we refer to the monograph \cite{MarinucciPeccati} for further details on this subject. In order to simplify comparison with the existing literature, we adopt the notation from \cite{MarinucciPeccati}. Let $T=(T(x):x\in\SS^2)$ be a centred, isotropic random field on the two-dimensional unit sphere $\SS^2$ having finite moments up to order three. Later we additionally assume that $T$ is Gaussian. According to \cite[Theorem 5.13]{MarinucciPeccati}, for each $x\in\SS^2$, $T(x)$ admits the harmonic representation
\begin{equation}\label{eq:HarmonicRepresentation}
T(x)=\sum_{\ell=0}^\infty\sum_{m=-\ell}^\ell a_{\ell,m}\,Y_{\ell,m}(x)\,,
\end{equation}
where $(Y_{\ell,m}:\ell\geq 0,\,-\ell\leq m\leq\ell)$ are the spherical harmonics and $(a_{\ell,m}:\ell\geq 0,\,-\ell\leq m\leq\ell)$ is an array of random coefficients determined by the random field $T$. Abbreviating the inner sum in the above representation by $T_\ell(x)$ it holds that $\EE[T_\ell(x)^2]=C_\ell\,\frac{2\ell+1}{4\pi}$, independently of $x\in\SS^2$, and the sequence $(C_\ell:\ell\geq 0)$ is called the angular power spectrum of $T$ (see Proposition 6.6 and Equation (6.21) in \cite{MarinucciPeccati}). If $T$ is a Gaussian random field, the angular power spectrum completely captures the dependence structure of $T$. In the non-Gaussian case, this structure becomes more involved and an analysis of higher-order angular power spectra is necessary. As a third-order characteristic, one can consider the integrals
$$
\int_{\SS^2} \EE[T_{\ell_1}(x) T_{\ell_2}(x) T_{\ell_3}(x)] \, \dint x\,, \qquad \ell_1,\ell_2,\ell_3\geq 0\,.
$$
After evaluating these expressions and re-scaling in a suitable way, one obtains the angular (average) power bispectrum $(B_{\ell_1,\ell_2,\ell_3}:\ell_1,\ell_2,\ell_3\geq 0)$, which is given by
$$
B_{\ell_1,\ell_2,\ell_3}:=\sum_{m_1=-\ell_1}^{\ell_1}\sum_{m_2=-\ell_2}^{\ell_2}\sum_{m_3=-\ell_3}^{\ell_3}\left(\begin{matrix}\ell_1 & \ell_2 & \ell_3\\ m_1 & m_2 & m_3\end{matrix}\right)\EE[a_{\ell_1,m_1}a_{\ell_2,m_2}a_{\ell_3,m_3}]\,,\quad \ell_1,\ell_2,\ell_3\geq 0\,.
$$
Here $\left(\begin{matrix}\ell_1 & \ell_2 & \ell_3\\ m_1 & m_2 & m_3\end{matrix}\right)$ is a combinatorial coefficient only depending on $\ell_1,\ell_2,\ell_3$ and $m_1,m_2,m_3$, the so-called Wigner $3j$-coefficient for which we refer to \cite[Chapter 3.5.3]{MarinucciPeccati}. We remark that these coefficients are closely related to the Clebsch-Gordan coefficients, a commonly used tool in the representation theory of compact Lie groups. The Wigner $3j$-coefficients vanish unless $m_1+m_2+m_3= 0$ and $|\ell_i-\ell_j|\leq \ell_k \leq \ell_i+\ell_j$ for all $i,j,k\in\{1,2,3\}$ (see \cite[Proposition 3.44]{MarinucciPeccati}).

In the following let $\ell_1,\ell_2,\ell_3\geq 0$ be such that $\ell_1+\ell_2+\ell_3$ is even and $|\ell_i-\ell_j|\leq \ell_k \leq \ell_i+\ell_j$ for all $i,j,k\in\{1,2,3\}$. A high-frequency, unbiased, minimum mean square error estimator for the angular power bispectrum is given by
\begin{equation}\label{eq:widehatB}
\widehat{B}_{\ell_1,\ell_2,\ell_3}:=\sum_{m_1=-\ell_1}^{\ell_1}\sum_{m_2=-\ell_2}^{\ell_2}\sum_{m_3=-\ell_3}^{\ell_3}
\left(\begin{matrix}\ell_1 & \ell_2 & \ell_3\\ m_1 & m_2 & m_3\end{matrix}\right)\,a_{\ell_1,m_1}a_{\ell_2,m_2}a_{\ell_3,m_3}\,,
\end{equation}
where $a_{\ell_1,m_1},a_{\ell_2,m_2},a_{\ell_3,m_3}$ are the observed coefficients in the harmonic representation \eqref{eq:HarmonicRepresentation} of the given realization of $T$. We also define its normalized version $S_{\ell_1,\ell_2,\ell_3}:=\widehat{B}_{\ell_1,\ell_2,\ell_3}/\sqrt{C_{\ell_1}C_{\ell_2}C_{\ell_3}}$, cf.\ \cite[Chapter 9.2.2]{MarinucciPeccati}. The estimator $\widehat{B}_{\ell_1,\ell_2,\ell_3}$ is known as the sample bispectrum of $T$, and we refer to $S_{\ell_1,\ell_2,\ell_3}$ as the re-scaled sample bispectrum. Because of the symmetry we assume without loss of generality that $\ell_1\leq\ell_2\leq\ell_3$.

An important problem in the statistical investigation of spherical random fields is to test for (non-) Gaussianity. Since the sample bispectrum is a prominent test statistic, we assume from now on that $T$ is Gaussian and study the behaviour of $S_{\ell_1,\ell_2,\ell_3}$ under this assumption.

By Lemma 9.6 in \cite{MarinucciPeccati}, the re-scaled sample bispectrum is an element of the third Wiener chaos associated with the underlying Gaussian random field $T$, which we can assume to be generated by a standard Brownian motion on the unit interval. In particular, $\EE[S_{\ell_1,\ell_2,\ell_3}]=0$ and $\EE[S_{\ell_1,\ell_2,\ell_3}^2]=D_{\ell_1,\ell_2,\ell_3}$ with $D_{\ell_1,\ell_2,\ell_3}:=1+{\bf 1}\{\ell_1=\ell_2\}+{\bf 1}\{\ell_2=\ell_3\}+3{\bf 1}\{\ell_1=\ell_3\}$ (see \cite[Theorem 9.7]{MarinucciPeccati}). For our asymptotic investigations let $(u_n:n\geq 1)$ and $(v_n:n\geq 1)$ be non-negative integer-valued sequences such that $n\leq u_n \leq v_n\leq 2n$ and suppose that $n+u_n+v_n$ is even for all $n\geq 1$. We recall from \cite[Theorem 9.9]{MarinucciPeccati} that, as $n\to\infty$, the normalized sample bispectrum of $T$ is asymptotically normal and that
$$
d_{TV}\left(\sqrt{\frac{6}{D_{n,u_n,v_n}}}\,S_{n,u_n,v_n},N\right)\leq \sqrt{\frac{32}{3n}}\,, \qquad n\geq 1\,,
$$
where $N$ is a Gaussian random variable with variance $3!=6$. Using the theory developed in Section \ref{sec:Results}, we can add a moderate deviation principle as well as certain moderate and large deviation estimates for the normalized sample bispectrum.

\begin{theorem}
Let $(u_n:n\geq1)$ and $(v_n:n\geq 1)$ be non-negative integer-valued sequences such that $n\leq u_n\leq v_n\leq 2n$ and $n+u_n+v_n$ is even for all $n\geq 1$ and let $F_n:=\sqrt{\frac{6}{D_{n,u_n,v_n}}}S_{n,u_n,v_n}$. Then the statements of Theorem \ref{thm:MDP} and Theorem \ref{thm:LDIs} hold with $q=3$ and $\Delta_n= 3^{-9/2}(\sqrt{3n}/2)^{5/12}$, $n\geq 1$.
\end{theorem}
\begin{proof}
Since the normalized sample bispectrum is an element of the third Wiener chaos, there exists a sequence $(h_{n,u_n,v_n}:n\geq 1)$ of square-integrable and symmetric functions on $[0,1]^3$ (supplied with the Lebesgue measure) such that $\|h_{n,u_n,v_n}\|_{L^2([0,1]^3)}=1$ and
$$
{\sqrt{\frac{6}{D_{n,u_n,v_n}}}}\,S_{n,u_n,v_n}=I_3^W(h_{n,u_n,v_n})
$$
for all $n\geq 1$, where $I_3^W$ stands for a multiple stochastic integral of order three with respect to a standard Brownian motion $W$ on $[0,1]$. It has been shown in the proof of Theorem 9.7 in \cite{MarinucciPeccati} that the fourth cumulant $\cum_4(I_3^W(h_{n,u_n,v_n}))$ of $I_3^W(h_{n,u_n,v_n})$ is bounded from above by $432/n$. Consequently, recalling the definition \eqref{eq:DefinitionLn}, we find that
$$
L_n=\frac{\sqrt{\cum_4(I_3^W(h_{n,u_n,v_n}))}}{3\cdot 3!}\leq \frac{1}{3\cdot 3!}\sqrt{\frac{432}{n}}=\frac{2}{\sqrt{3n}}\,,
$$
and hence the conditions of Corollary \ref{cor:MDP} and Theorem \ref{thm:LDIs} (iii) are satisfied with $q=3$ and
$$
\Delta_n = 3^{-9/2}\bigg(\frac{\sqrt{3n}}{2}\bigg)^{5/12}\,,$$
since $\alpha(3)=5/12$. This proves the claim.
\end{proof}

\section{Proofs of the main results}\label{sec:Proofs}

The next lemma is our main device to prove Theorem \ref{thm:MDP} and Theorem \ref{thm:LDIs}. It summarizes a moderate deviation principle and fine probability estimates, which are available under certain bounds on cumulants. This approach goes back to the `Lithuanian school of probability', and we refer especially to the monograph \cite{SaulisBuch}.

\begin{lemma}\label{lem:CumulantBoundImplyMDPandConcentration}
Let $(X_n:n\geq 1)$ be a sequence of real-valued random variables such that $\EE[X_n]=0$, $\EE[X_n^2]=\sigma^2\geq 1$ and $\EE[|X_n|^m]<\infty$ for all $m\geq 1$. Suppose that there is a constant $\gamma\geq 0$ such that the cumulants of $X_n$ satisfy
\begin{equation}\label{eq:CumumantBound}
|\cum_m(X_n)|\leq \frac{(m!)^{1+\gamma}}{\Delta_n^{m-2}}\qquad\text{for all}\quad m\geq 3
\end{equation}
with $\Delta_n>0$ for $n\geq 1$.
\begin{itemize}
\item[(a)]  Let $(a_n:n\geq 1)$ be a real sequence such that $$\lim_{n\to\infty}a_n=\infty\qquad{\rm and}\qquad\lim_{n\to\infty}{a_n/\Delta_n^{1/(1+2\gamma)}}=0\,.$$ Then the sequence of re-scaled random variables $(a_n^{-1}X_n:n\geq 1)$ satisfies a MDP with speed $a_n^2$ and Gaussian rate function $\cI(z)=z^2/(2\sigma^2)$.
\item[(b)] There exist constants $c_0,c_1,c_2>0$ only depending on $\gamma$ such that for $\D_n\geq c_0$ and $0\leq z\leq c_1\sigma\D_n^{1/(1+2\g)}$,
$$
\bigg|\log\frac{\PP(X_n\geq z)}{1-\Phi_{\sigma^2}(z)}\bigg|\leq c_2 \frac{1+(z/\sigma)^3}{\D_n^{1/(1+2\g)}}\qquad{\rm and}\qquad\bigg|\log\frac{\PP(X_n\leq -z)}{\Phi_{\sigma^2}(-z)}\bigg|\leq c_2 \frac{1+(z/\sigma)^3}{\D_n^{1/(1+2\g)}}\,,
$$
where $\Phi_{\sigma^2}$ is the distribution function of  a centred Gaussian random variable with variance $\sigma^2$.
\item[(c)] One has that
$$
\PP(|X_n|\geq z)\leq 2\exp\left(-\frac{1}{4}\min\Big\{\frac{z^2}{2^{1+\gamma}},\,(z\D_n)^{1/(1+\g)}\Big\}\right)
$$
for all $z\geq 0$.
\end{itemize}
\end{lemma}
\begin{proof}
Define $\widetilde{X}_n:=X_n/\sigma$ and observe that by $\sigma^2\geq 1$ and \eqref{eq:CumumantBound},
$$
|\cum_m(\widetilde{X}_n)| = \sigma^{-m}\,|\cum_m(X_n)| \leq |\cum_m(X_n)| \leq \frac{(m!)^{1+\gamma}}{\Delta_n^{m-2}}\,.
$$
Applying \cite[Theorem 1.1]{EichelsbacherDoering} and \cite[Corollary 3.2]{EichelsbacherSchreiberRaic} to $(\widetilde{X}_n:n\geq 1)$ yields part (a) and part (b). The assertion in (c) follows from the Corollary after Lemma 2.4 in \cite{SaulisBuch} with $H=2^{1+\gamma}$ there. As discussed in Remark \ref{rem:DiscussionWoher} above, the results in (b) and (c) are simplified versions of the `main lemmas' from \cite{RudzkisSaulisStatu} and \cite{BentkusRudzkis}, respectively, which are summarized in Chapter 2 of \cite{SaulisBuch}.
\end{proof}

\begin{remark}\label{rem:Weibull}\rm
\begin{itemize}
\item[(i)] To require the estimate \eqref{eq:CumumantBound} is a rather natural condition from the viewpoint of complex analysis. As discussed at the beginning of Chapter 2 in \cite{SaulisBuch}, the cumulant bound \eqref{eq:CumumantBound} with $\gamma=0$ implies analyticity of the cumulant generating functions $\log\EE[\exp(wX_n)]$, $w\in\CC$, of the random variables $(X_n:n\geq 1)$ in the discs $\{|w|\leq\Delta_n\}\subset\CC$. On the other hand, if the cumulant generating function is analytic in a disc around the origin with radius $\Delta_n$ and such that
    $$
    \sup\limits_{|w|=\Delta_n}\big|\log\EE[\exp(wX_n)]\big|\leq\Delta_n^2\,,
    $$
    then the random variable $X_n$ satisfies the cumulant bound \eqref{eq:CumumantBound} by Cauchy's integral formula for derivatives (recall \eqref{eq:DefCumulant}).

To allow for $\gamma>0$ takes into account a heavy tail behaviour and a super-exponential growth of cumulants (or moments) of the involved random variables. For example, if a random variable has density $x\mapsto \frac{\alpha}{2}|x|^{\alpha-1}e^{-|x|^\alpha}$ for some $\alpha>0$, then the $m$th moment is $\Gamma(1+m/\alpha)$ if $m$ is even and zero otherwise. Thus, Stirling's formula together with \cite[Lemma 3.1]{SaulisBuch} implies that the cumulant bound \eqref{eq:CumumantBound} is satisfied with $\gamma=\frac{1}{\alpha}-1$ if $0<\alpha\leq 1$ and $\gamma=0$ if $\alpha>1$. In view of the tail estimate \eqref{eq:Major} it is therefore not unexpected that we have $\gamma=\frac{q}{2}-1$ for random variables belonging to the $q$th Wiener chaos $\cC_q$, $q\geq 2$. We also refer to Remark \ref{rem:lowerBound} below for further discussion of this point.

\item[(ii)] Under condition \eqref{eq:CumumantBound} one also has the Berry-Esseen estimate
\begin{equation}\label{eq:BerryEsseen}
\sup_{x\in\RR}\big|\PP(X_n\leq x)-\Phi_{\sigma^2}(x)\big|\leq c\,\Delta_n^{-{1/(1+2\g)}}
\end{equation}
with a constant $c>0$ only depending on $\gamma$, see \cite[Corollary 2.1]{SaulisBuch}. In the context of Theorem \ref{thm:MDP} this leads to a rate of convergence of order $K_n^{\alpha(q)/(q-1)}$. To the best of our knowledge, this provides a first proof of the fourth moment theorem including rates of convergence without resorting to Stein's method. However, comparing this with the bound \eqref{eq:RateAusDemBuch} derived via the Malliavin-Stein method in \cite{NourdinPeccati09} (see also \cite[Chapter 5.2]{NourdinPeccatiBook}), we see that the rate of convergence via the method of cumulants is weaker for all $q\geq 2$. Moreover, the bound in \eqref{eq:RateAusDemBuch} is for the total variation distance, which is larger than the left-hand side of \eqref{eq:BerryEsseen}. For this reason, we do not pursue rates for the normal approximation further in this paper.
\end{itemize}
\end{remark}

Our strategy for the proof of Theorem \ref{thm:MDP} and Theorem \ref{thm:LDIs} is to establish for the random variable $F_n=I_q(h_n)$ the cumulant bound \eqref{eq:CumumantBound}. In what follows we assume without loss of generality that $\cH=L^2(A,\sA,\mu)=:L^2(A)$ with a Polish space $(A,\mathcal{A})$ and a non-atomic $\sigma$-finite measure $\mu$. This is possible because of isomorphy of Hilbert spaces. Recall that we denote by $L^2_s(A^n)$, $n\geq 1$, the subspace of $L^2(A^n)$ consisting of symmetric functions, i.e., functions which are invariant under permutation of their arguments. Moreover, the tensor product of two functions $f_1: A^{n_1}\to\RR$ and $f_2:A^{n_2}\to\RR$, $n_1,n_2\geq 1$, is a function $f_1\otimes f_2: A^{n_1+n_2}\to\RR$ given by
$$
f_1\otimes f_2(x_1,\hdots,x_{n_1+n_2})=f(x_1,\hdots,x_{n_1}) f(x_{n_1+1},\hdots,x_{n_1+n_2})\,.
$$

\medskip

For integers $\ell\geq 1$ and $n_1,\ldots,n_\ell\geq 1$ define $N_0:=0$ and, for $i\in\{1,\ldots,\ell\}$, put $N_i:=n_1+\ldots+n_i$ and $J_i:=\{N_{i-1}+1,\ldots,N_i\}$. By a partition $\sigma$ of $\{1,\ldots,N_\ell\}$ we understand a collection $\{B_1,\ldots,B_{|\sigma|}\}$ of non-empty and pairwise disjoint subsets $B_1,\ldots,B_{|\sigma|}\subset\{1,\ldots,N_\ell\}$ with $B_1\cup\ldots\cup B_{|\sigma|}=\{1,\ldots,N_\ell\}$. The sets $B_1,\ldots,B_{|\sigma|}$ are called blocks, and by $|\sigma|$ we denote the number of blocks of $\sigma$. Let $\Pi(n_1,\ldots,n_\ell)$ be the set of partitions $\sigma$ of $\{1,\ldots,N_\ell\}$ satisfying
\begin{itemize}
\item $|B\cap J_i|\leq 1$ for all $i\in\{1,\ldots,\ell\}$ and blocks $B\in\sigma$,
\item $|B|=2$ for all blocks $B$ of $\sigma$,
\item for all non-empty sets $M_1,M_2\subset\{1,\ldots,\ell\}$ with $M_1\cup M_2=\{1,\ldots,\ell\}$ there are a block $B\in\sigma$ and elements $i_1\in M_1$ and $i_2\in M_2$ such that $B\cap J_{i_1}\neq\emptyset$ and $B\cap J_{i_2}\neq\emptyset$.
\end{itemize}
For brevity, we also write $\Pi(q[m],q_1,\ldots,q_k)$ instead of $\Pi(q,\ldots,q,q_1,\ldots,q_k)$, where $q$ appears $m$ times. Note that $\Pi(n_1,\ldots,n_\ell)$ can be empty, in particular if $N_\ell$ is odd.

For a function $f:A^{N_\ell}\to\RR$ and a partition $\sigma\in\Pi(n_1,\ldots,n_\ell)$ we define $f_\sigma:A^{|\sigma|}\to\RR$ by replacing all arguments of $f$ with indices belonging to the same block of $\sigma$ by a new common variable. For example, if $f:A^4\to\RR$ and $\sigma=\{\{1,4\},\{2,3\}\}$, we have that $f_\sigma(y,z)=f(y,z,z,y)$. This notation allows us to recall from \cite[Corollary 7.3.1]{PeccatiTaqquBook} (see also Proposition 5.6 in \cite{NourdinPeccatiCumulants}) the following classical expression for the cumulants of a random variable of the type $I_q(h)$ with $h\in L_s^2(A^q)$.

\begin{lemma}\label{lem:Kumulanten}
For $q\geq 1$ and $h\in L_s^2(A^q)$,
$$\cum_m(I_q(h))=\sum_{\sigma\in\Pi(q[m])}\int_{A^{|\sigma|}}(h^{\otimes m})_\sigma\,\dint\mu^{|\sigma|}\,,\qquad m\geq 1\,,$$
where the sum on the right-hand side has to be interpreted as $0$ if $\Pi(q[m])=\emptyset$.
\end{lemma}

A crucial step in our proof of Theorem \ref{thm:MDP} and Theorem \ref{thm:LDIs} is to re-write the right-hand side of the cumulant expression provided in Lemma \ref{lem:Kumulanten} in terms of contractions. The underlying idea is to consider for a fixed partition $\sigma\in\Pi(q[m])$ the functions in the tensor product $h^{\otimes m}$ appearing on the right-hand side in Lemma \ref{lem:Kumulanten} as vertices of a suitable multigraph induced by $\sigma$ and to group them in an appropriate way according to a maximal matching.

More precisely, with $\sigma\in\Pi(q[m])$ we associate a multigraph $G_\sigma$ as follows. The set of vertices is $\{1,\ldots,m\}$, and for each block $B$ of $\sigma$ with $B=(B\cap J_i)\cup (B\cap J_j)$ we connect the vertices $i$ and $j$ by an edge. In particular, each vertex of $G_\sigma$ has degree $q$ (in other words this means that $G_\sigma$ is $q$-regular), and $G_\sigma$ is a connected multigraph without loops.

A matching of $G_\sigma$ is a set of non-adjacent edges of $G_\sigma$, and we denote by $M(G_\sigma)$ the maximal size of such a set. This so-called matching number of $G_\sigma$ is an important quantity considered in combinatorics. For the matching number of a multigraph $G_\sigma$ associated with a partition $\sigma\in\Pi(q[m])$ we have the following lower bound.

\begin{lemma}\label{lem:MatchingNumber}
For $q\geq 2$ and $m\geq 3$ suppose that $\Pi(q[m])\neq\emptyset$ and fix $\sigma\in\Pi(q[m])$. Then
\begin{align}\label{eq:LowerBoundLqm}
M(G_\sigma)\geq L(q,m) := \begin{cases} \Big\lceil \frac{(q^2-q-1)m-(q-1)}{q(3q-5)}\Big\rceil &: \text{$q$ odd}\\ \min\Big(\Big\lfloor\frac{m}{2}\Big\rfloor,\Big\lceil\frac{(q+2)m}{3q+2}\Big\rceil\Big) &: \text{$q$ even}\,.\end{cases}
\end{align}
\end{lemma}
\begin{proof}
From Theorem 1 (with $\lambda=2$ there) and Theorem 2 in \cite{Nishizeki} it follows that $M(G_\sigma)$ satisfies the inequality $M(G_\sigma) \geq L(q,m)$ for $q\geq 3$. If $q=2$, $G_\sigma$ is a cycle and hence also satisfies the estimate \eqref{eq:LowerBoundLqm}.
\end{proof}

The construction of the multigraph $G_\sigma$ described above and the lower bound on the matching number in Lemma \ref{lem:MatchingNumber} allow us to re-write the summands on the right-hand side of the cumulant expression in Lemma \ref{lem:Kumulanten} in the following way:

\begin{lemma}\label{lem:CrucialEstimate}
For $q\geq 2$ and $m\geq 3$ suppose that $\Pi(q[m])\neq\emptyset$ and fix $\sigma\in\Pi(q[m])$ and $h\in L_s^2(A^q)$. Then there are non-negative integers $m_1,m_2$ satisfying $m_1+2m_2=m$ and $L(q,m)\leq m_2\leq m/2$ and $r_i\in\{1,\ldots,q-1\}$ for $i\in\{1,\ldots,m_2\}$ as well as a partition $\widetilde{\sigma}\in\Pi(q[m_1],q-r_1,\ldots,q-r_{m_2})$ such that
$$\int_{A^{|\sigma|}}(h^{\otimes m})_\sigma\,\dint\mu^{|\sigma|} = \int_{A^{|\widetilde{\sigma}|}}\Big(h^{\otimes m_1}\otimes\bigotimes_{i=1}^{m_2}(h\otimes_{r_i}h)\Big)_{\widetilde{\sigma}}\,\dint\mu^{|\widetilde{\sigma}|}\,.$$
\end{lemma}
\begin{proof}
 We construct the multigraph $G_\sigma$ as described above and choose a matching of maximal cardinality. Then we split the set of variables of $(h^{\otimes m})_\sigma$ into two groups, those variables belonging to vertices (factors in the tensor product $h^{\otimes m}$) that are matched and the remaining variables. Now, we apply Fubini's theorem and integrate over the first group of variables. By this construction, $m_2=M(G_\sigma)$ pairs of functions $h$ are transformed into terms of the type $h\otimes_{r_i}h$ with $r_i\in\{1,\ldots,q-1\}$, $i=1,\ldots,m_2$. In fact, that $r_i<q$ follows since $G_\sigma$ is connected by construction. We write the integration with respect to the variables belonging to the second group in terms of a partition $\widetilde{\sigma}\in\Pi(q[m-2M(G_\sigma)],q-r_1,\ldots,q-r_{m_2})$ and set $m_1:=m-2M(G_\sigma)$. Now the observation from Lemma \ref{lem:MatchingNumber} that $m_2=M(G_\sigma)\geq L(q,m)$ concludes the proof.
\end{proof}

\begin{remark}\label{rem:SharpLqm}\rm
It follows from the examples in Section 3 in \cite{Nishizeki} that for sufficiently large $m$ there are partitions $\sigma$ such that inequality \eqref{eq:LowerBoundLqm} is sharp in that $M(G_\sigma)=L(q,m)$, and hence $m_2=L(q,m)$ and $m_1=m-2L(q,m)$ in Lemma \ref{lem:CrucialEstimate}.
\end{remark}

Finally, let us recall the following generalized Cauchy-Schwarz inequality from Lemma 4.1 in \cite{BiermeEtAl}.
\begin{lemma}\label{lem:GeneralizedCauchySchwarz}
Fix $\ell\geq 2$, $n_1,\ldots,n_\ell\geq 1$ such that $\Pi(n_1,\ldots,n_\ell)\neq\emptyset$, $f_i\in L^2_s(A^{n_i})$ for $i\in\{1,\ldots,\ell\}$ and $\sigma\in\Pi(n_1,\ldots,n_\ell)$. Then
$$\int_{A^{|\sigma|}}\Big(\bigotimes_{i=1}^\ell |f_i|\Big)_\sigma\,\dint\mu^{|\sigma|}\leq \prod_{i=1}^\ell \|f_i\|_{L^2(A^{n_i})}\,.$$
\end{lemma}

\begin{remark}\label{rem:SharpCSU}\rm
If $B\in\sA$ with $\mu(B)<\infty$ and each of the functions $f_i$ is of the form $$f_i(x_1,\ldots,x_{n_i})=\prod\limits_{j=1}^{n_i}{\bf 1}(x_j\in B)\,,$$ we have equality in Lemma \ref{lem:GeneralizedCauchySchwarz}.
\end{remark}

After these preparations, we can now establish the cumulant bound \eqref{eq:CumumantBound} for the random variable $F_n=I_q(h_n)$.

\begin{proof}[Proof of Theorem \ref{thm:MDP} and Theorem \ref{thm:LDIs}]
Using the cumulant formula provided in Lemma \ref{lem:Kumulanten} we have that
\begin{align*}
\cum_m(F_n) &=\sum_{\sigma\in\Pi(q[m])}\int_{A^{|\sigma|}}(h_n^{\otimes m})_\sigma\,\dint\mu^{|\sigma|}\,.
\end{align*}
To each summand we apply Lemma \ref{lem:CrucialEstimate} with the notation introduced there and Lemma \ref{lem:GeneralizedCauchySchwarz} to see that
\begin{equation}\label{eq:CumXXX}
\begin{split}
\Big|\int_{A^{|\sigma|}}(h_n^{\otimes m})_\sigma\,\dint\mu^{|\sigma|}\Big| &\leq  \int_{A^{|\widetilde{\sigma}|}}\Big(|h_n|^{\otimes m_1}\otimes\bigotimes_{i=1}^{m_2}|h_n\otimes_{r_i}h_n|\Big)_{\widetilde{\sigma}}\,\dint\mu^{|\widetilde{\sigma}|}\\
&\leq \|h_n\|_{L^2(A^q)}^{m_1}\,\prod_{i=1}^{m_2}\|h_n\otimes_{r_i}h_n\|_{L^2(A^{2(q-r_i)})}\\
&\leq K_n^{L(q,m)}
\end{split}
\end{equation}
with $K_n$ defined at \eqref{eq:defKn}. Here, we have used that $L(q,m)\leq m_2$, the assumption that $\|h_n\|_{L^2(A^q)}=1$ and that $K_n\leq 1$. The latter property is a consequence of the Cauchy-Schwarz inequality, implying that $\|h_n\otimes_rh_n\|_{L^2(A^{2(q-r)})}^2\leq\|h_n\|_{L^2(A^q)}^4=1$ for all $r\in\{1,\ldots,q-1\}$.
Thus, \eqref{eq:CumXXX} yields that
\begin{align*}
|\cum_m(F_n)| \leq |\Pi(q[m])|\,K_n^{L(q,m)}\,.
\end{align*}
It has been shown in Proposition 5.3 of \cite{SaulisBuch} that $|\Pi(q[m])|$ is bounded from above by
\begin{equation}\label{eq:BoundCqm}
|\Pi(q[m])|\leq (m!)^{q/2}(q^{q/2})^m\,.
\end{equation}
Moreover, it follows from the definition of $L(q,m)$ that $L(q,m)\geq\alpha(q)(m-2)$ with $\alpha(q)$ defined at \eqref{eq:defAlphaq}. Consequently,
\begin{align*}
|\cum_m(F_n)| &\leq (m!)^{q/2}(q^{q/2})^m\,K_n^{\alpha(q)(m-2)}\leq (m!)^{q/2} (q^{3q/2})^{m-2}\,K_n^{\alpha(q)(m-2)}\,,
\end{align*}
where we have used that $3(m-2)\geq m$ for $m\geq 3$. Choosing
$$\gamma=\frac{q}{2}-1\quad\text{and}\quad\Delta_n^{-1}=q^{3q/2}\,K_n^{\alpha(q)}$$
establishes the cumulant bound \eqref{eq:CumumantBound}. In view of Lemma \ref{lem:CumulantBoundImplyMDPandConcentration} this concludes the proof of Theorem \ref{thm:MDP} and Theorem \ref{thm:LDIs}.
\end{proof}

\begin{remark}\label{rem:lowerBound}\rm
\begin{itemize}
\item[(i)] Proposition 5.3 in \cite{SaulisBuch} also gives the lower bound
\begin{equation}\label{eq:LBoundCqm}
|\Pi(q[m])|\geq \frac{1}{8}(m!)^{q/2}(\sqrt{2})^m\,.
\end{equation}
Comparison of \eqref{eq:BoundCqm} and \eqref{eq:LBoundCqm} shows that we can in general choose $\gamma$ not smaller than $\frac{q}{2}-1$. This goes hand in hand with the observations made in Remark \ref{rem:Weibull} above.

In contrast to our situation, it is typically a difficult task to decide whether for given random variables the parameter $\gamma$ in the cumulant estimate \eqref{eq:CumumantBound} is optimal or not. Such a situation arises, for example, in \cite{EichelsbacherSchreiberRaic}, where problems from geometric probability have been considered. There,  $\gamma$ is different from zero and depends on the particular model and even on the space dimension.

\item[(ii)] Because of the sharpness of the lower bound on the matching number and the sharpness of the generalized Cauchy-Schwarz inequality discussed in Remarks \ref{rem:SharpLqm} and \ref{rem:SharpCSU}, there are situations for which one has equality in all estimates of \eqref{eq:CumXXX}. This shows that the exponent of $K_n$ in $\Delta_n$ is optimal in general.
\end{itemize}
\end{remark}

We finally establish Proposition \ref{prop:SumHermite} based on arguments from the proof of Theorem 6.12 in \cite{Janson}.

\begin{proof}[Proof of Proposition \ref{prop:SumHermite}]
Since $F_n$ is a linear combination of $H_q(X(e_k))$, $k\in\{1,\hdots,n\}$, it is by definition an element of the $q$th Wiener chaos $\cC_q$, and we have that $F_n=I_q(h_n)$ with $h_n=\frac{1}{\sqrt{n}} \sum_{k=1}^n e_k^{\otimes q}$ for $n\geq 1$. Hence,
$$
\Var F_n=\EE[I_q(h_n)^2]=q!\|h_n\|^2_{\cH^{\otimes q}}=q!
$$
and
$$
K_n=\max_{r=1,\hdots,q-1}\|h_n\otimes_r h_n\|_{\cH^{\otimes 2(q-r)}}=\max_{r=1,\hdots,q-1}\frac{1}{n} \Big\|\sum_{k=1}^n e_k^{\otimes 2(q-r)}\Big\|_{\cH^{\otimes 2(q-r)}}=\frac{1}{\sqrt{n}}\,.
$$
For $n\geq 1$ define $S_n=\sum_{k=1}^n H_q(X(e_k))$ and note that $S_2$ can be regarded as a polynomial of degree $q$ depending on the random variables $X(e_1)$ and $X(e_2)$. Thus, it follows from Equation (6.10) in \cite{Janson} that there are constants $\widetilde{C},\widetilde{c},\widetilde{t}_0>0$ such that
\begin{equation}\label{eq:BoundS2}
\PP(|S_2|\geq t) \geq \widetilde{C} \exp(-\widetilde{c} t^{2/q}) \qquad \text{ for all } \qquad t\geq \widetilde{t}_0 \,.
\end{equation}
Since $(S_n:n\geq 1)$ is a martingale with respect to the natural filtration induced by the random variables $(X(e_k): k\geq 1)$ we have that
\begin{equation}\label{eq:BoundConditional}
\EE[ |S_n| \,|\, X(e_1),X(e_2) ] \geq |S_2|\qquad\qquad\PP\text{-a.s.}
\end{equation}
for $n\geq 2$. Hence,
$$
\PP(|S_n|\geq u) \geq \PP(|S_n|\geq |S_2|/2, |S_2|\geq 2u) \geq \PP(|S_n|\geq \EE[|S_n| \,|\, X(e_1),X(e_2)]/2, |S_2|\geq 2u) \,,
$$
where we used \eqref{eq:BoundConditional} for the second inequality.  After re-writing, we get
$$
\PP(|S_n|\geq u) \geq \EE[\, \EE[ {\bf 1}\{|S_n| \geq \EE[|S_n| \,|\, X(e_1),X(e_2)]/2\} \,|\, X(e_1),X(e_2)] \, {\bf 1}\{|S_2|\geq 2u\}\,]\,.
$$
If $X(e_1)$ and $X(e_2)$ are given, $S_n$ is the sum of an element in the $q$th Wiener chaos and a constant so that Theorem 6.9 in \cite{Janson} yields that
$$
 \EE[ {\bf 1}\{|S_n| \geq \EE[|S_n| \,|\, X(e_1),X(e_2)]/2\} \,|\, X(e_1),X(e_2)] \geq c_q\qquad\qquad\PP\text{-a.s.}
$$
with a constant $c_q>0$ only depending on $q$. We thus obtain that
$$
\PP(|S_n|\geq u) \geq c_q \, \PP(|S_2|\geq 2u) \,.
$$
Now, the estimate \eqref{eq:BoundS2} yields, for $z\geq z_0:=\widetilde{t}_0$, that
$$
\PP( |F_n| \geq z)=\PP( |S_n| \geq \sqrt{n}z) \geq c_q \widetilde{C}\, \exp(-2^{2/q}\widetilde{c}\, n^{1/q} z^{2/q}) \,,
$$
which proves \eqref{eq:LowerBoundHermite}.

Since $a_n\to\infty$, as $n\to\infty$, it follows from \eqref{eq:LowerBoundHermite} that, for any $t>0$,
\begin{align*}
\limsup_{n\to\infty} a_n^{-2} \log \PP(a_n^{-1} F_n\geq t) & \geq \limsup_{n\to\infty} a_n^{-2} \log( C \exp(-c\, n^{1/q} (a_n t)^{2/q}))\\
& = \limsup_{n\to\infty} \frac{-c n^{1/q} (a_n t)^{2/q}}{a_n^2}\,.
\end{align*}
If $n^{1/(2q-2)}/a_n\to 0$ as $n\to\infty$, the right-hand side converges to zero, implying that $(a_n^{-1} F_n:n\geq1)$ cannot satisfy a MDP with speed $a_n^2$ and Gaussian rate function.
\end{proof}

\subsection*{Acknowledgements}
Parts of this paper were written during a Research-in-Pairs stay of the authors at Mathema\-tisches Forschungsinstitut Oberwolfach. All support is gratefully acknowledged. We also thank Sabine Jansen for a stimulating discussion.

MS has been funded by the German Research Foundation (DFG) through the research unit ``Geome\-try and Physics of Spatial Random Systems" under the grant HU 1874/3-1. CT has been supported by the German research foundation (DFG) via SFB-TR 12.


\end{document}